\documentclass[12pt]{article}

\usepackage[T1]{fontenc} \usepackage[english]{babel}
\usepackage{latexsym, amsmath, amssymb, amsthm, color, xcolor,graphicx}
\usepackage{enumitem}
\usepackage{authblk}
\newtheorem{theorem}{Theorem}[section]
\newtheorem{lemma}[theorem]{Lemma}
\newtheorem{corollary}[theorem]{Corollary}

\newtheorem{proposition}[theorem]{Proposition}
\theoremstyle{definition}
\newtheorem{definition}[theorem]{Definition}
\newtheorem{example}[theorem]{Example}
\newtheorem*{remark}{Remark}
\newtheorem{construction}[theorem]{Construction}

\newcommand{\Z}{\mbox{${\bf Z}$}}

\newcommand{\aut}{\mbox{${\rm Aut}$}}

\newcommand{\autga}[0]{\mbox{${\rm Aut}(\Gamma)$}}

\newcommand{\out}{\mbox{${\rm out}$}}
\newcommand{\inn}{\mbox{${\rm in}$}}

\newcommand{\ZZ}{\mathbb{Z}}

\newcommand{\Aut}{\mathrm{Aut}}

\newcommand{\V}{\mathrm V}

\newcommand{\E}{\mathrm E}
\newcommand{\A}{\mathrm A}
\newcommand{\R}{\mathrm R}

\newcommand{\Sym}{\mathrm{Sym}}

\newcommand{\mv}{\mbox{${\rm mv}$}}

\title{Trivalent vertex-transitive graphs with infinite vertex-stabilizers}
\author[1]{Arnbj\"org Soff\'ia \'Arnad\'ottir}
\author[2]{Waltraud Lederle}
\author[3]{R\"ognvaldur G. M\"oller\footnote{e-mail: roggi@raunvis.hi.is. }}

\affil[1]{University of Waterloo}
\affil[2]{Universit\'e Catholique de Louvain}
\affil[3]{University of Iceland}

\begin{document}

\maketitle

\begin{abstract}
We study groups acting vertex-transitively on connected, trivalent graphs such that stabilizers of vertices are infinite.
If the action is edge-transitive, we prove that the graph has to be a tree.
We analyze the case where the action is not edge-transitive and fully classify the possible $2$-ended graphs.
We draw connections to Willis' scale function and re-prove a result by Trofimov.
\end{abstract}




\section*{Introduction}

Tutte's papers on trivalent graphs in 1947, \cite{Tutte1947},  and 1959, \cite{Tutte1959}, are rightly regarded as the starting point of the study of group actions on graphs as a separate discipline and his ideas in these two papers still today have deep and profound influences.  Tutte investigates arc-transitive group actions on finite, connected, trivalent graphs.  In many of his results the assumption that the graph is finite can be dropped and replaced with the assumption that the stabilizer of a vertex is a finite group, see \cite{DjokovicMiller1980}.  For instance, if $\Gamma$ is a connected trivalent graph and $G$ acts arc-transitively on $\Gamma$ and vertex stabilizers are finite, then $G$ acts regularly on the set of $s$-arcs for some $s\leq 5$.   

The aim in this work is to study the \lq\lq other\rq \rq\ case; i.e.\ vertex-transitive subgroups of the automorphism group of a trivalent, connected graph such that the  stabilizer of a vertex is an infinite group.  It turns out that insisting on infnite vertex stabilzers has a drastic influence on the graph.

Tutte's idea to study the action of the group on $s$-arcs is also fundamental in our work.  Using Tutte's methods we prove:

\smallskip

{\bf Corollary 3.4.}
{\em Let $\Gamma$ be a connected, trivalent graph.  Suppose $G\leq \aut(\Gamma)$ acts  
vertex- and edge-transitively on $\Gamma$ and assume the stabilizers in $G$ of vertices in $\Gamma$ are infinite.  Then $\Gamma$ is a 3-regular tree.}

\smallskip 

But the 3-regular tree is not the only example of a connected trivalent graph such that the automorphism group is vertex-transitive and stabilizers of vertices are infinite.  In these other cases the automorphism group has two orbits on the edges.  Here the key is to study the action of the group on $s$-arcs where the edges come alternatively from the two orbits.  Several examples of such graphs are described in Section 4.
We give a full classification of such graphs with only $2$ ends, see Theorem~\ref{thm:2-ends}.
A general classification of possible graphs seems difficult.

Locally finite, vertex-transitive graphs are tightly connected to totally disconnected, locally compact groups via the Cayley--Abels graph.
A Cayley--Abels graph of a compactly generated, totally disconnected, locally compact group is an analogue of an ordinary Cayley graph for a finitely generated group.  Our results can be applied to the study of  compactly generated, totally disconnected, locally compact groups that have a trivalent Cayley--Abels graph.

\smallskip

{\bf Corollary 6.3.}
{\em Let $G$ be a compactly generated, totally disconnected, locally compact group having a trivalent Cayley--Abels graph.
If every $g \in G$ normalizes a compact, open subgroup of $G$ (i.e.~$G$ is uniscalar), then $G$ has a compact, open, normal subgroup.}

\smallskip

A further application is a short proof of the following theorem of Trofimov.

\smallskip

{\bf Theorem 7.2.}{\rm (\cite[Theorem 3.1]{Trofimov1984})}
{\em Let $\Gamma$ be a vertex transitive trivalent graph and $G=\aut (\Gamma)$.  Then $G$ has a compact normal subgroup $N$ such that the stabilizers in $\aut (\Gamma/N)$ of vertices in $\Gamma/N$ are finite or $\Gamma_2$ contains a subgraph isomorphic to the 3-regular tree.}
\smallskip

Here $\Gamma_2$ denotes the graph one gets by adding to $\Gamma$ all edges of the type $\{\alpha, \beta\}$, where the distance between $\alpha$ and $\beta$ is 2.

\section{Notation and preliminary remarks}

\subsection{Graphs}
The graphs we consider have neither loops nor multiple edges. Thus an (undirected) graph $\Gamma$ can be defined as a pair $(\V\Gamma, \E\Gamma)$, where $\V\Gamma$ is the set of {\em vertices} and $\E\Gamma$, the set of {\em edges}, is a set of two element subsets of $\V\Gamma$. We define the set of {\em arcs}, $\A\Gamma$, of $\Gamma$ as the set of all ordered pairs $(\alpha, \beta)$ such that $\{\alpha, \beta\}\in \E\Gamma$.  
Two vertices $\alpha$ and $\beta$ are said to be {\em adjacent}, or {\em neighbours}, if $\{\alpha, \beta\}$ is an edge.  The {\em degree} of a vertex in a graph is the cardinality of its set of neighbours.  A graph is said to be {\em regular} if all vertices have the same degree $d$, and then we say that $d$ is the degree of the graph.
A graph is {\em locally finite} if the degree of every vertex is finite.  

 We also consider digraphs (directed graphs). A \emph{digraph} consists of a vertex set $\V\Gamma$ and a subset  $\A\Gamma \subseteq \V\Gamma \times \V\Gamma$ that does not intersect the diagonal.  The elements of $\V\Gamma$ are called vertices and the elements of $\A\Gamma$ are called {\em arcs}.  
The \emph{underlying undirected graph} of a digraph $\Gamma$ has the same vertex set as $\Gamma$ and the set of edges is the set of all pairs $\{\alpha, \beta\}$ where $(\alpha, \beta)$ or $(\beta, \alpha)$ is an arc in $\Gamma$.   
For a vertex $\alpha$ in a digraph $\Gamma$ 
we define the sets of {\em in-} and {\em out-neighbours} as
$\inn(\alpha)=\{\beta\in \V\Gamma\mid (\beta,\alpha)\in \A\Gamma\}$ and 
$\out(\alpha)=\{\beta\in \V\Gamma\mid (\alpha,\beta)\in \A\Gamma\}$, respectively.
The cardinality of $\inn(v)$ is the {\em in-degree} of $v$ and the cardinality of $\out(v)$ is the {\em out-degree} of $v$.   A digraph is {\em regular} if any two vertices have the same in-degree and also the same out-degree.  

For an integer $s \geq 0$ an \emph{$s$-arc} in $\Gamma$ (a digraph or an undirected graph) is a $(s+1)$-tuple $(\alpha_0,\dots,\alpha_s)$ of vertices such that for every $0 \leq i \leq s-1$ the pair $(\alpha_{i},\alpha_{i+1})$ is an arc in $\Gamma$, and $\alpha_{i-1} \neq \alpha_{i+1}$ for all $1 \leq i \leq s-1$.   Infinite arcs come in three different shapes.  There are  1-way infinite arcs,  $(\ldots, \alpha_{-1}, \alpha_0)$ and $(\alpha_0, \alpha_1, \ldots)$, and then there are 2-way infinite arcs $(\ldots, \alpha_{-1}, \alpha_0, \alpha_1, \ldots)$.  In all cases we insist that $(\alpha_{i},\alpha_{i+1})$ is an arc in $\Gamma$, and $\alpha_{i-1} \neq \alpha_{i+1}$ for all $i$.

A {\em path} of length $s\geq 0$ in a graph $\Gamma$ is a subgraph with vertex set $\{\alpha_0, \ldots, \alpha_s\}$, the vertices $\alpha_0, \ldots, \alpha_s$ are distinct,  and edge set $\{\{\alpha_0, \alpha_1\},\ldots, \{\alpha_{s-1}, \alpha_s\}\}$.  The vertices $\alpha_0$ and $\alpha_s$ are called the {\em end-vertices} of the path and we speak of an $\alpha_0-\alpha_s$ path. 
Paths can also be infinite.  A {\em ray} in a graph $\Gamma$ is a subgraph with vertex set $\{\alpha_0, \alpha_1, \ldots \}$ and edge set $\{\{\alpha_0, \alpha_1\},\{\alpha_1, \alpha_2\}, \ldots\}$ such that all the vertices $\alpha_0, \alpha_1, \ldots$ are distinct. A {\em line} is a subgraph with vertex set $\{\ldots, \alpha_{-1}, \alpha_0, \alpha_1, \ldots \}$ and edge set $\{\ldots, \{\alpha_{-1}, \alpha_0\},\{\alpha_0, \alpha_1\},\{\alpha_1, \alpha_2\}, \ldots\}$ such that all the vertices $\ldots, \alpha{-1}, \alpha_0, \alpha_1, \ldots$ are distinct.   We often refer to paths, rays and lines by listing the vertices in the natural order.  Thus a path $P$ with vertex set $\{\alpha_0, \ldots, \alpha_s\}$ and edge set $\{\{\alpha_0, \alpha_1\},\ldots, \{\alpha_{s-1}, \alpha_s\}\}$ will be denoted by $P=\alpha_0, \ldots, \alpha_s$ and similarly for rays and lines.  

We say a sequence $\alpha_0, \ldots, \alpha_s$ is a path in a digraph $\Gamma$ if the vertices $\alpha_0, \ldots, \alpha_s$ are all distinct and $(\alpha_i, \alpha_{i+1})$ or $(\alpha_i, \alpha_{i+1})$ is an arc for all $0\leq i\leq s-1$.  Equivalently, a sequence $\alpha_0, \ldots, \alpha_s$ is a path if and only if it is a path in the underlying undirected graph. Rays and lines in digraphs are defined analogously.


We say that $\Gamma$ is {\em connected} if for every pair of vertices $\alpha$ and $\beta$ in $\Gamma$ there exists an $\alpha-\beta$ path in $\Gamma$.  The {\em distance} between vertices $\alpha$ and $\beta$ in a connected graph is defined as the length of a shortest $\alpha-\beta$ path and is denoted with $d_\Gamma(\alpha, \beta)$.  A digraph is connected if its underlying undirected graph is connected and the distance between two vertices in a connected digraph is the same as the distance between the corresponding vertices in the underlying undirected graph.

An {\em end} of a graph $\Gamma$ is an equivalence class of rays: two rays $R_1$ and $R_2$ in $\Gamma$ are said to be equivalent if there is a third ray $R_3$ that intersects both $R_1$ and $R_2$ in infinitely many vertices.  In the special case when the graph $\Gamma$ is a tree then two rays belong to the same end if and only if their intersection is a ray.  The set of ends of $\Gamma$ is denoted with $\Omega\Gamma$.   When $\Gamma$ is a digraph we define the ends of $\Gamma$ in terms of the ends of the underlying undirected graph.  

\subsection{Groups}

Let $G$ be a group acting (on the right) on a set $\Omega$.  
Denote the image of a point $\alpha\in \Omega$ under an element $g\in G$ by $\alpha g$.  The action is said to be {\em transitive} if for any two points $\alpha, \beta$ in $\Omega$ there exists an element $g\in G$ such that $\alpha g=\beta$.  The {\em stabilizer} of $\alpha\in \Omega$ is the subgroup $G_\alpha=\{g\in G\mid \alpha g=\alpha\}$.   For a set $A\subseteq \Omega$  the {\em pointwise stabilizer} of $A$ is the subgroup
$G_{(A)}=\{g\in G\mid \alpha g=\alpha\mbox{ for all }\alpha \in A\}$.   The kernel of the action is the subgroup $K=\{g\in G\mid \alpha g=\alpha \mbox{ for all }\alpha\in \Omega\}$.  When $K=\{1\}$ we say that the action is {\em faithful} and then we can think of $G$ as a permutation group of $\Omega$, i.e.\ a subgroup of $\Sym(\Omega)$, the group of all symmetry group of the set $\Omega$.   

An action of a group $G$ on a set $\Omega$  is called \emph{semi-regular} (or free) if $G_\alpha=\{1\}$ for all points $\alpha\in \Omega$ and \emph{regular} if it is semi-regular and transitive.

 A \emph{graph morphism} between two graphs (or digraphs) $\Gamma$ and $\Delta$ is a map $\varphi \colon \V\Gamma \to \V\Delta$ such that if $(\alpha, \beta) \in \A\Gamma$ then $(\varphi(\alpha),\varphi(\beta)) \in \A\Delta$.
  If $\Gamma$ is a graph or a digraph and $\varphi:\V\Gamma\to\V\Gamma$ is a bijective map, then $\varphi$ is an \emph{automorphism} of $\Gamma$ if $\varphi$ induces a bijection $\A\Gamma\to \A\Gamma$.  The set of all automorphisms of $\Gamma$ is a group, the {\em automorphism group of $\Gamma$},  denoted by $\Aut(\Gamma)$.  We will think of $\Aut(\Gamma)$ and subgroups of $\Aut(\Gamma)$ as permutation groups on $\V\Gamma$.
  
  A graph or a digraph $\Gamma$ is {\em vertex-transitive} if the automorphism group acts transitively on the vertex set.  Vertex-transitive graphs are always regular.  We say that $\Gamma$ is {\em edge-transitive} or {\em arc-transitive} if the automorphism group acts transitively on the edge set or arc set, respectively.    If the automorphism group of $\Gamma$ acts transitively on the set of $s$-arcs in $\Gamma$ then we say that $\Gamma$ is {\em $s$-arc-transitive}.  When the automorphism group is $s$-arc-transitive for all $s$,  we say that $\Gamma$ is {\em highly-arc-transitive}. 

Consider now a group $G$ that acts vertex-transitively on a graph $\Gamma$ of degree $d$. Let $\alpha \in \V\Gamma$.
The stabilizer $G_\alpha$ clearly leaves $N(\alpha)$, the set of neighbours of $\alpha$, invariant and thus induces an action on it. The kernel of this action is $K_\alpha=G_\alpha\cap G_{(N(\alpha))}$ and the quotient $G_\alpha / K_\alpha$ is a subgroup of $\Sym(d)$. Let now $\alpha'$ be another vertex of $\Gamma$. By assumption there exists $g \in G$ with $\alpha g = \alpha'$. The actions of $G_\alpha$ on $N(\alpha)$ and $G_{\alpha'}$ on $N(\alpha)$ are conjugate via $g$. Thus, the following is independent of the choice of $\alpha$.

\begin{definition}
  Let $\Gamma$ be a graph of degree $d$ on which a group $G$ acts vertex-transitively.
  Let $\alpha \in \V\Gamma$.
  The \emph{local action} of $G$ on $\Gamma$ is the conjugacy class of the finite group $G_{\alpha}/K_\alpha$, seen as a subgroup of $\Sym(d)$.
\end{definition}
Usually we will say that the local action is the subgroup $G_{\alpha}/K_\alpha$ of $\Sym(d)$ and omit the mention of the conjugacy class.

When $\sigma$ is an equivalence relation on the vertex set of a graph $\Gamma$ we can form the {\em quotient graph}  $\Gamma/\sigma$. Its vertex set is the set of $\sigma$-classes, and if $A$ and $B$ are distinct $\sigma$-classes then $\{A,B\}$ is an edge in $\Gamma/\sigma$ if and only if there is a vertex $\alpha\in A$ and a vertex $\beta\in B$ such that $\{\alpha,\beta\}$ is an edge in $\Gamma$.  If $G$ is a subgroup of $\autga$ then $\Gamma/G$ denotes the quotient graph of $\Gamma$ with respect to the equivalence relation whose classes are the $G$-orbits on the vertex set.  If $\sigma$ is a $G$-congruence (i.e.~$\alpha g$ is equivalent to $\beta g$ if and only if $\alpha$ is equivalent to $\beta$) then $G$ has a natural action on the $\sigma$-classes and thus an action on the quotient graph $\Gamma/\sigma$ by automorphisms.  Quotients of digraphs are defined in the obvious way.  

A faithful action of a group $G$ on a set $\Omega$ is said to be {\em discrete} if the stabilizers of vertices are finite.  If the action is discrete it is possible to find a finite subset $A \subset \Omega$ such that $G_{(A)}=\{1\}$.  

\subsection{Convergent sequences of permutations}\label{sec:Convergence}

In this section the notions of {\em convergence of sequences of permutations} and and {\em closed groups of permutations } are introduced.  Here we avoid actually introducing a topology, but in Section~\ref{sec:Scale} we will see a group topology on a permutation group such that the convergence we introduce here is convergence in that topology.

\begin{definition}
  Let $\{g_i\}$ be a sequence of permutations of some set $\Omega$.  We say that the sequence converges to a permutation $g$ of $\Omega$ if for every point $\alpha\in \Omega$ there exists a number $N_\alpha \geq 0$ such that $\alpha g_i=\alpha g$ for all $i\geq N_\alpha$.
  
  A group $G$ of permutations of some set $\Omega$ is said to be {\em a closed permutation group} (or a {\em closed} subgroup of $\Sym(\Omega)$) if, whenever $\{g_i\}$ is a sequence of permutations in $G$ converging to a permutation $g$ of $\Omega$, then $g\in G$.
\end{definition}

It is easy to show that the automorphism group of a graph (or a digraph) $\Gamma$ is closed.  It is also easy to see that if the action is discrete then every convergent sequence is eventually constant.  
The following lemma will be used in Section~\ref{sec:three-cases} and is the reason why these terms are introduced here.  

\begin{lemma}{\rm (Cf.\ \cite[Lemma 1]{Moller2002a})}\label{lem:infinte-arcs}
Let $\Gamma$ be a locally finite, connected graph (or digraph) and assume that $G$ is a closed subgroup of $\aut(\Gamma)$.  Suppose $G$ acts highly-arc-transitively on $\Gamma$.  Then $G$ acts transitively on the set of 2-way infinite arcs of $\Gamma$. In particular, $G$ acts transitively on the set of 1-way infinite arcs of $\Gamma$ of type $(\ldots, \alpha_{-1}, \alpha_0)$ and on the set of 1-way infinite arcs of type $(\alpha_0, \alpha_1, \ldots)$.
\end{lemma}

\begin{proof}
Let  $(\ldots, \alpha_{-1}, \alpha_0, \alpha_1, \ldots)$ and  $(\ldots, \beta_{-1}, \beta_0, \beta_1, \ldots)$ denote two 2-way infinite arcs in $\Gamma$.  Since $G$ acts highly-arc-transitively on $\Gamma$, there is for each $i\geq 0$ an element $g_i$ such that $(\alpha_{-i}, \ldots, \alpha_i)g_i=(\beta_{-i}, \ldots, \beta_i)$.  Let $A_i$ denote the set of all vertices in $\Gamma$ at distance at most $i$ from $\alpha_0$.  Because the graph $\Gamma$ is assumed to be locally finite, the sets $A_i$ are all finite.  All elements in the sequence $\{g_i\}$ map the vertex $\alpha_0$ to the vertex $\beta_0$.  Since $A_1$ is finite, there are only finitely many possibilities for the maps we get by restricting the $g_i$'s to $A_1$.  Hence there is an infinite set $C_1$ of elements from the sequence $\{g_i\}$ such that the restriction of all these elements to $A_1$ is the same.  Let $i_1$ be a number such that $g_{i_1}$ is in $C_1$.  There are also only finitely many possibilities for the restriction of the permutations in the sequence $\{g_i\}$ to $A_2$ and thus we get an infinite subset $C_2$ of $C_1$ such that restrictions of the elements in $C_2$ to $A_2$ are all identical.  Choose $i_2$ such that $i_2>i_1$ and $g_{i_2}$ is in $C_2$.  Continuing in this way we get a subsequence $\{g_{i_j}\}$ of our original sequence so that if $j, j'\geq i$ and $\alpha$ is a vertex in $A_i$ then $\alpha g_{i_j}=\alpha g_{i_{j'}}$.  Hence we can define a permutation $g$ of the vertex set of $\Gamma$ by saying that $\alpha g$ is equal to $\alpha g_{i_j}$ for $j$ equal to the distance in $\Gamma$  between $\alpha_0$ and $\alpha$.  It is now clear that the sequence $\{g_{i_j}\}$ converges to $g$ and that $(\ldots, \alpha_{-1}, \alpha_0, \alpha_1, \ldots)g=(\ldots, \beta_{-1}, \beta_0, \beta_1, \ldots)$.   The assumption that $G$ is a closed permutation group guarantees that $g\in G$.  This shows that $G$ acts transitively on the set of 2-way infinite arcs.  

In a highly-arc-transitive graph a 1-way infinite arc of either type can always be extended to a 2-way infinite arc and thus the statement about the transitivity of the action on the sets of 1-way infinite arcs follows from the transitivity of the action on 2-way infinite arcs.
\end{proof}

\section{Three cases}\label{sec:three-cases}

The first step in our investigation is to use the local action and the action on the edges and arcs to divide non-discrete vertex-transitive group actions on trivalent graphs into three cases.  

Suppose $G$ is a vertex-transitive subgroup of the automorphism group of some connected trivalent graph $\Gamma$.
The local action of $G$ on $\Gamma$ is a conjugacy class of subgroups of the symmetric group $S_3$. There are four possibilities: the trivial group, the cyclic group of order $3$, the cyclic group of order $2$ and the whole group $S_3$.   
First note that if the stabilizer of a vertex acts locally like the trivial group then, since $\Gamma$ is connected, we see that the stabilizer of a vertex acts trivially on the graph.  In the case where the stabilizer of a vertex acts locally like a cyclic group of order 3 we see similarly that the subgroup fixing some pair of adjacent vertices is trivial.  In both cases the stabilizer of a vertex is a finite group and the action is discrete.  

If $G$ is non-discrete then we are left with the possibilities that the group acts locally either like the full symmetric group or like the cyclic group of order 2.  In the first case it is clear that  the group acts both edge- and arc-transitively on $\Gamma$. 

Assume that $G$ acts locally like a cyclic group of order 2.  It is possible that the group $G$ is edge-transitive, and then $G$ is not arc-transitive,  but it is also possible $G$ has  two orbits on the edges of $\Gamma$. Let us briefly analyse the latter case.  

Let $\alpha$ be a vertex of $\Gamma$ and let $\beta$ denote the neighbour of $\alpha$ that is fixed by $G_\alpha$.
Edges in the $G$-orbit of $\{\alpha, \beta\}$  will be called {\em red} and the edges in the other edge-orbit will be called {\em blue}.  Each vertex in $\Gamma$ is therefore the end-vertex of precisely one red edge and precisely two blue edges and this colouring is preserved by the action of $G$.   We say that arcs in $\Gamma$ inherit a colour from the edge that gives rise to them. 

Remove all the blue edges from $\Gamma$.  As each vertex is the end vertex of only one red edge, we get a vertex-transitive graph of degree 1. We see that $G$ must act transitively on the red arcs, in particular there exists an element $g\in G$ such that $\alpha g=\beta$ and $\beta g=\alpha$.  Removing the red edges from $\Gamma$ we get a vertex-transitive graph of degree 2.  Each connected component is therefore either a finite cycle or a line and the connected components are all isomorphic. Since the stabilizer of a vertex acts locally like the cyclic group of order 2 we see that $G$ acts transitively on the blue arcs.  Hence the group has two orbits on the arcs of $\Gamma$.  

\medskip

The outcome of the above discussion is that when we have a connected trivalent graph $\Gamma$ and a non-discrete subgroup $G\leq \aut(\Gamma)$ acting vertex-transitively, then there are three possible cases:  

\medskip

Case {\bf A}:  The stabilizer of a vertex is infinite and the group acts locally like the symmetric group on three elements.  The group acts transitively on both the set of edges and the sets arcs of $\Gamma$.  

\smallskip
Case {\bf B}:  The stabilizer of a vertex is infinite and acts locally like a cyclic group of order two.  The group acts transitively on the edges, but is not transitive on the arcs.  

\smallskip 

Case {\bf C}:  The stabilizer of a vertex is infinite and acts locally like a cyclic group of order two and the group has two orbits on the edges and two orbits on the arcs. We call edges that are fixed by the local action {\bf red} and other edges {\bf blue}.

\medskip

Before continuing to analyse these cases let us look at examples.

\begin{example}\label{ex:the-cases} $ $

\begin{enumerate}
\item The regular 3-valent tree $T_3$ and its automorphism group are an example of a graph satisfying the conditions in Case {\bf A}.  It is shown in the next section that in Case {\bf A} the graph $\Gamma$ must be the 3-valent tree.
\item Let $\Gamma=T_3$ and let $\Gamma_+$ be the digraph we get if we orient the edges of $\Gamma$ so that at each vertex there is one incoming arc and two outgoing arcs.  The action of $\aut(\Gamma_+)$ on $\Gamma$ satisfies the conditions in Case {\bf B}.  In the next section it is shown that in Case {\bf B} the graph $\Gamma$ is equal to $T_3$.
\item  Let $\Gamma=T_3$ denote the 3-regular tree.  Colour each edge red or blue so that each vertex is adjacent to one red edge and two blue edges.  Let $G$ denote the subgroup of the automorphism group of $\Gamma$ that preserves this colouring.  The stabilizer in $G$ of a vertex $\alpha$ is infinite and the action of $G$ on $\Gamma$ satisfies the conditions in Case {\bf C}.  
\item  The arc-graph $A_1(\Gamma)$ of a graph $\Gamma$ has as its vertex set the set of arcs of $\Gamma$ and two arcs $(\alpha, \beta)$ and $(\gamma, \delta)$ are adjacent in $A_1(\Gamma)$ if and only if $\beta=\delta$ or both $\alpha=\delta$ and $\beta=\gamma$.   Each vertex in $T_3$ gives rise to a triangle in $A_1(T_3)$ and the triangles for a pair of adjacent vertices are joined by a single edge.  This graph is trivalent and if we colour the edges in the triangles blue and the other edges red then we have the situation described in Case {\bf C}.   Clearly it is possible to join $n$-gons in a similar way to get a trivalent graph resembling the $n$-regular tree that also satisfies the condition in Case {\bf C}.
\item  Start with a $2n$-gon.  For each pair $\alpha, \beta$ of opposite vertices in that $2n$-gon take a new $2n$-gon and select some pair $\delta, \gamma$ of opposite vertices in the new $2n$-gon.  Now add edges $\{\alpha, \delta\}$ and $\{\beta, \gamma\}$.  Then look at pairs of opposite vertices in the new $2n$-gons where the vertices have degree 2 and for each such pair get a new $2n$-gon.  Continue like this {\em ad infinitum} until you have got a 3-regular graph (see Figure \ref{fig:trivalent2}).  The stabilizer of a vertex in the automorphism group of this new graph is clearly infinite and we have an example of Case {\bf C}.  Contracting each and everyone of the $2n$-gons leaves us with the $n$-regular tree.
\begin{figure}
    \centering
    \includegraphics[scale=0.5]{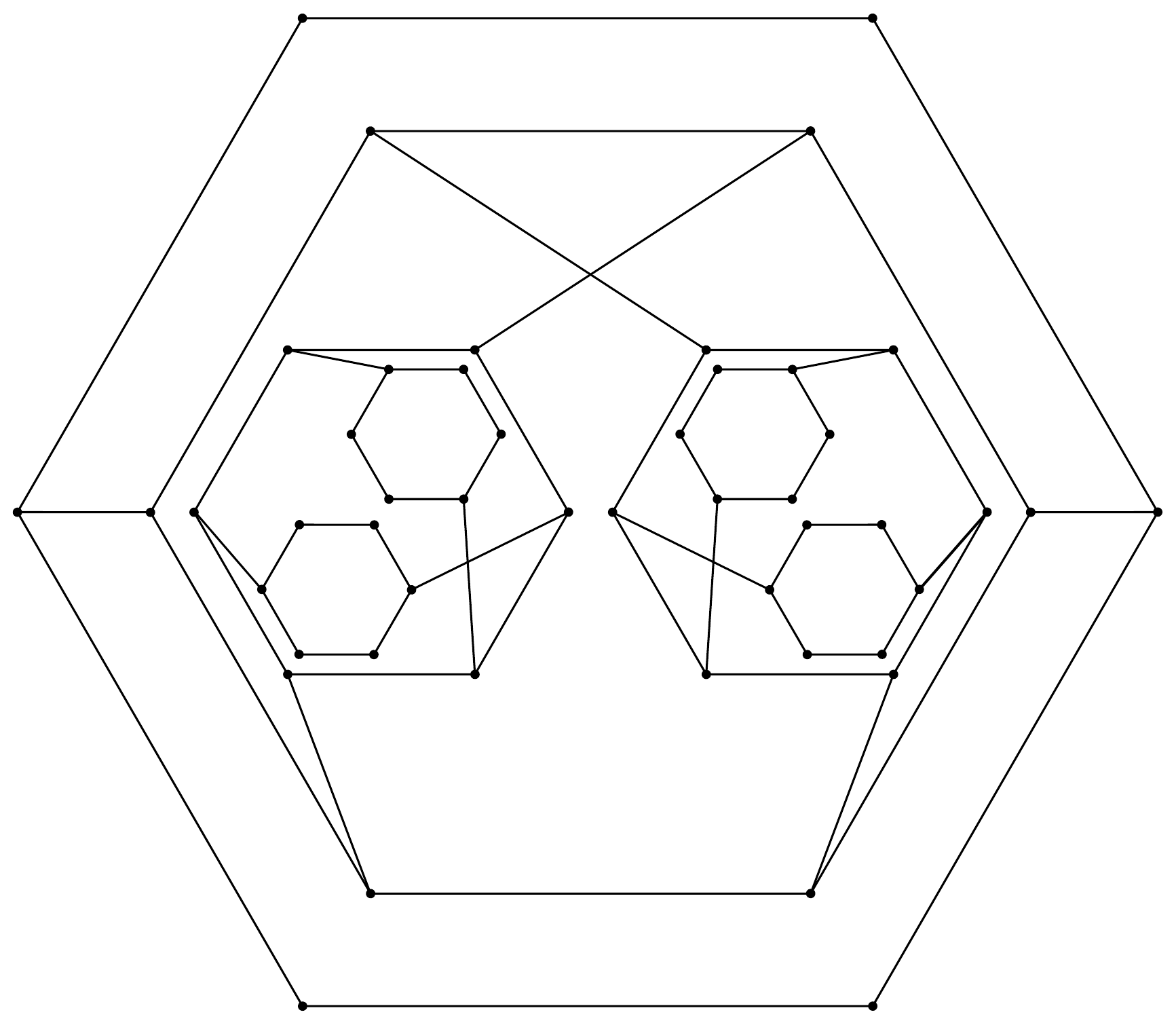}
    \caption{An example of Case {\bf C}}
    \label{fig:trivalent2}
\end{figure}
\item Suppose $\Delta$ is a connected digraph such that both the in- and out-degrees of all vertices are equal to 2.
We construct a trivalent graph by essentially replacing every vertex of $\Delta$ by an arc pointing from the two incoming to the two outgoing arcs.
Formally, let $\Delta_*$ be defined as follows. For each vertex $\alpha$ in $\Delta$ we put two vertices $\alpha_-$ and $\alpha_+$ in the vertex set of $\Delta_*$. The arc set of $\Delta_*$ consists of all pairs $(\alpha_-, \alpha_+)$ and all pairs $(\alpha_+, \beta_-)$, where $(\alpha, \beta)$ is an arc in $\Delta$ (see Figure \ref{fig:DL}). 
\begin{figure}
    \centering
\begingroup%
  \makeatletter%
  \providecommand\color[2][]{%
    \errmessage{(Inkscape) Color is used for the text in Inkscape, but the package 'color.sty' is not loaded}%
    \renewcommand\color[2][]{}%
  }%
  \providecommand\transparent[1]{%
    \errmessage{(Inkscape) Transparency is used (non-zero) for the text in Inkscape, but the package 'transparent.sty' is not loaded}%
    \renewcommand\transparent[1]{}%
  }%
  \providecommand\rotatebox[2]{#2}%
  \newcommand*\fsize{\dimexpr\f@size pt\relax}%
  \newcommand*\lineheight[1]{\fontsize{\fsize}{#1\fsize}\selectfont}%
  \ifx\svgwidth\undefined%
    \setlength{\unitlength}{131.30867618bp}%
    \ifx\svgscale\undefined%
      \relax%
    \else%
      \setlength{\unitlength}{\unitlength * \real{\svgscale}}%
    \fi%
  \else%
    \setlength{\unitlength}{\svgwidth}%
  \fi%
  \global\let\svgwidth\undefined%
  \global\let\svgscale\undefined%
  \makeatother%
  \begin{picture}(1,0.75881501)%
    \lineheight{1}%
    \setlength\tabcolsep{0pt}%
    \put(0,0){\includegraphics[width=\unitlength,page=1]{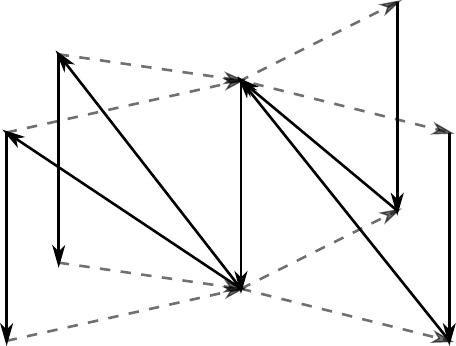}}%
    \put(0.50098888,0.61772477){\color[rgb]{0,0,0}\makebox(0,0)[lt]{\lineheight{1.25}\smash{\begin{tabular}[t]{l}$\alpha_-$\end{tabular}}}}%
    \put(0.49727241,0.0290364){\color[rgb]{0,0,0}\makebox(0,0)[lt]{\lineheight{1.25}\smash{\begin{tabular}[t]{l}$\alpha_+$\end{tabular}}}}%
  \end{picture}%
\endgroup%

    \caption{The graph $\Delta_*$}
    \label{fig:DL}
\end{figure}
If $g$ is an automorphism of $\Delta$ then we define $g_*: \V\Delta_*\to \V\Delta_*$ by setting $\alpha_-g_*=(\alpha g)_-$ and $\alpha_+g_*=(\alpha g)_+$ and it is clear that $g_*$ is an automorphism of $\Delta_*$. Note that, in general, there does not exist an automorphism of $\Delta_*$ mapping $\alpha_-$ to $\alpha_+$, so we require an extra condition.
The {\em reverse digraph} $\Delta^R$ of $\Delta$ has the same vertex set as $\Delta$ and $(\alpha, \beta)$ is an arc in $\Delta^R$ if and only if $(\beta, \alpha)$ is an arc in $\Delta$.  The extra condition we are imposing is that $\Delta$ is isomorphic to its reverse digraph $\Delta^\R$ via some graph isomorphism $f: \Delta\to \Delta^\R$.  Define $f_*: \V\Delta_*\to \V\Delta_*$ by setting $\alpha_-f_*=(\alpha f)_+$ and $\alpha_+f_*=(\alpha f)_-$.  Clearly $f_*$ is an automorphism of the undirected graph $\Gamma$ underlying $\Delta_*$.  Set $H=\{g_*\mid g\in \aut(\Delta)\}$ and $G=\langle H, f_*\rangle$.  The group $G$ is a group of automorphisms of the undirected graph $\Gamma$ and the graph $\Gamma$ is trivalent.  If the stabilizer in $\aut(\Delta)$ of a vertex in $\Gamma$ is infinite, e.g.\ if $\Delta$ is highly-arc-transitive,  then stabilizers of vertices in $G$ are infinite and we have an action that satisfies the conditions in Case {\bf C}.

An example of a highly-arc-transitive digraph $\Delta$ like the one described is the digraph with vertex set $\Z\times \{0,1\}$, where $((i,j), (i', j'))$ is an arc if and only if $i'=i+1$.  The construction then gives the graph $\Gamma$ from Figure \ref{fig:trivalent1}.  This graph has two ends.
\begin{figure}
    \centering
    \includegraphics[scale=0.5]{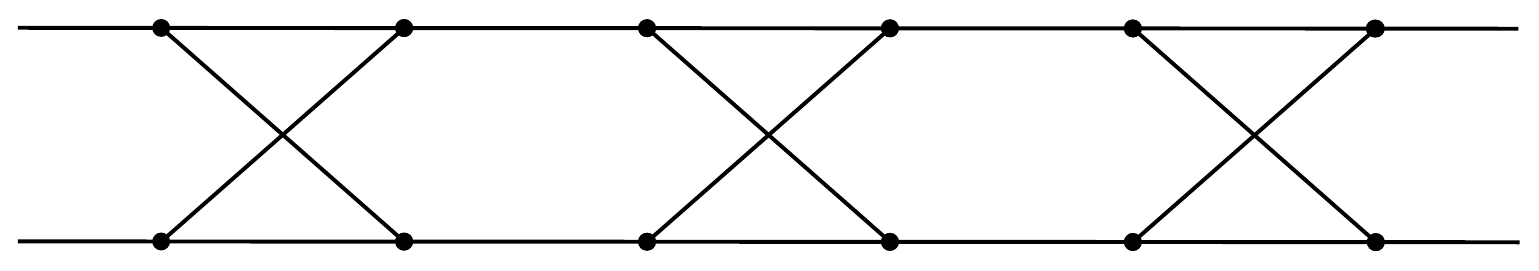}
    \caption{A 2-ended graph satisfying Case {\bf C}}
    \label{fig:trivalent1}
\end{figure}

If we take $\Delta$ as the digraph constructed in Example 1 in \cite{Moller2002a} we get an example of a trivalent graph that satisfies the conditions in Case {\bf C} but has only one end.  (The underlying undirected graph of $\Delta$ is the Diestel-Leader graph $DL(2,2)$. Diestel-Leader graphs have been discussed by various authors, see e.g.~\cite{DiestelLeader2001} and \cite{Woess2005}.)  The construction in Example 1 in \cite{Moller2002a} can be adapted to provide more examples of highly-arc-transitive digraphs that can be used in the above construction.
\end{enumerate}
\end{example}

The following proposition provides further examples of graphs and groups satisfying the conditions in Case {\bf C}.  

\begin{proposition}\label{prop:locally-dihedral}
Let $G$ be a group acting vertex-transitively on a connected, locally finite graph $\Gamma$ of degree $d$.
Suppose that the local action of $G$ on $\Gamma$ is the dihedral group with $2d$ elements in its natural action on a set with $d$ elements. Then, there is a connected trivalent graph $\Delta$ satisfying the following.
\begin{enumerate}
    \item The group $G$ has a vertex-transitive action on $\Delta$.
    \item If the action of $G$ on $\Gamma$ is not discrete, then the action on $\Delta$ is not discrete and satisfies the conditions of Case {\bf C}.
    \item There is a $G$-congruence $\sigma$ on $\V\Delta$ such that $\Gamma=\Delta/\sigma$ and the subgraph in $\Delta$ spanned by each $\sigma$ class is a $d$-gon.
\end{enumerate}
\end{proposition}

\begin{proof}
Define a graph $\Delta$ as follows:  The vertex set is the set of arcs of $\Gamma$.   Two arcs in $\Gamma$ (i.e.\ vertices in $\Delta$) are connected by a red edge if they are reverse to each other.   For a fixed vertex $\alpha$ in $\Gamma$ we choose an element $r_\alpha \in G_\alpha$ so that $r_\alpha$ acts on $N(\alpha)$, the neighbourhood of $\alpha$, as a $d$-cycle.  If a vertex $\beta$ is adjacent to $\alpha$ then we say that $\{(\alpha,\beta), (\alpha,\beta r_\alpha)\}$ is a blue edge in $\Delta$ and so are all the elements in the $G$-orbit of $\{(\alpha,\beta), (\alpha,\beta r_\alpha)\}$.  (Loosely speaking we can say that $\Delta$ is the graph we get if we replace each vertex in $\Gamma$ with a $d$-cycle with blue edges and then connect  $d$-cycles corresponding to adjacent vertices with red edges.)

Clearly the graph $\Delta$ we constructed is a trivalent graph on which $G$ acts vertex-transitively.  If the action on $\Gamma$ is not discrete then the action of $G$ on $\Delta$ is not discrete and satisfies the conditions in Case {\bf C}.
\end{proof}

The construction in the proof of Proposition~\ref{prop:locally-dihedral} is \lq\lq reversible\rq\rq:  Suppose $G$ acts 
vertex-transitively on a trivalent graph $\Delta$ such that the conditions in Case {\bf C} are satisfied.   Consider the \lq\lq blue  subgraph\rq\rq\ with the same vertex set as $\Gamma$ and edge set the set of blue edges. This subgraph is $2$-regular, so there is a number $d\geq 3$ such that each connected component is a $d$-gon  for some $d \geq 3$ or every component is a line. In the first case, by contracting in $\Gamma$ each  connected component of the blue subgraph to a vertex, we get a vertex-transitive graph on which $G$ acts  and the local action is a dihedral group. In the second case we could also contract each blue line to a vertex and get a graph on which $G$ acts vertex-transitively.  But in this case the graph would not be locally finite and the group would act locally like the infinite dihedral group.  

\begin{remark}
In \cite{Nebbia2013} Nebbia studies non-discrete, vertex-transitive actions on the 3-regular tree and describes the same division into cases as above.
\end{remark}

\section{Vertex- and edge-transitive actions}

In this section we show that in Cases {\bf A} and {\bf B} the graph $\Gamma$ is the 3-regular tree.  This is proved by the methods used by Tutte in \cite{Tutte1959}.

\begin{construction}
The following construction is used in the proofs of Theorem~\ref{thm:trivalentA} and \ref{thm:trivalentB}, and also in subsequent sections of this paper.  The {\em arc-digraph} of $\Gamma$ (sometimes called the {\em line graph} or the {\em parital line graph}) is denoted with $B_1(\Gamma)$.  The set of vertices is the set of arcs in $\Gamma$, i.e.\ $\V B_1(\Gamma)=\A\Gamma$,  and if $(\alpha, \beta), (\gamma, \delta)$ are arcs in $\Gamma$ then $((\alpha, \beta), (\gamma, \delta))$ is an arc in $B_1(\Gamma)$ if and only if $\beta=\gamma$.  The $s$-arc-digraph $B_s(\Gamma)$ is defined such that
the set of vertices of $B_s(\Gamma)$ is the set of $s$-arcs of $\Gamma$ and
the arcs in $B_s(\Gamma)$ are  pairs $((\alpha_0,\dots,\alpha_{s}),(\alpha_1,\dots,\alpha_{s+1}))$, where $(\alpha_0,\dots,\alpha_{s+1})$ is a $(s+1)$-arc in $\Gamma$.
It is easy to see that if $s\geq 2$ then $B_s(\Gamma)=B_1(B_{s-1}(\Gamma))
$.
Note  that if $\Gamma$ is connected then $B_1(\Gamma)$ is also connected and by induction one sees that $B_s(\Gamma)$ is connected.
\end{construction}

We now prove our main theorem regarding Case {\bf A}.

\begin{theorem}\label{thm:trivalentA}  Suppose a group $G\leq\aut(\Gamma)$ acts vertex- and arc-transitively on a connected trivalent graph $\Gamma$ and that the stabilizers of vertices are infinite.  Then $\Gamma$ is a tree.

Furthermore, $G$ acts transitively on the set of $s$-arcs for any $s\geq 1$ and if $G$ is a closed subgroup of $\aut(\Gamma)$, then $G$ acts 2-transitively on the ends  of $\Gamma$.
\end{theorem}

\begin{proof}   
This result can be proved by adapting Tutte's proof of (2.2) from \cite{Tutte1959}, but instead we give a proof that uses the concept of a $s$-arc digraph defined above and the method of the proof of Lemma 4.3.2 in \cite{GodsilRoyle2001}.
Let $\Gamma_+$ be a digraph with vertex set $\V\Gamma$ and arc set $\A\Gamma$, so if $\{\alpha, \beta\}$ is an edge in $\Gamma$ then both $(\alpha, \beta)$ and $(\beta, \alpha)$ are arcs in $\Gamma_+$.  Our assumptions say that $G$ acts 1-arc-transitively on $\Gamma_+$.  Suppose there is some $s$ such that $G$ acts $s$-arc-transitively on $\Gamma_+$ but is not $(s+1)$-arc-transitive.  Consider the $s$-arc digraph $B_s(\Gamma_+)$.  This is a digraph where the in- and out-degrees are both equal to 2 and the natural action of $G$ on $B_s(\Gamma_+)$ is vertex transitive and faithful.  Let $P$ be a vertex in $B_s(\Gamma_+)$ corresponding to some $s$-arc $(\alpha_0, \ldots, \alpha_s)$ in $\Gamma_+$.  Denote with $P'=(\alpha_1, \ldots, \alpha_s, \alpha'_{s+1})$ and $P''=(\alpha_1, \ldots, \alpha_s, \alpha''_{s+1})$ the two predecessors of $P$ in $\Gamma_+$.  The $s$-arcs $P'$ and $P''$ are the two out-neighbours of $P$ in $B_s(\Gamma_+)$.  If the group $G_P$ contains an element $g$ that transposes $P'$ and $P''$ and $Q=(\beta_0.\ldots, \beta_s, \beta_{s+1})$ is some $(s+1)$-arc then we can find an element $h$ that takes the $s$-arc $(\beta_0.\ldots, \beta_s)$ to the $s$-arc $P$ and then $h$ takes $Q$ to either $P'$ or $P''$.  If $Qh=P''$ then $Qhg=P'$.  From this we conclude that $G$ must act transitively on the set of $(s+1)$-arcs contrary to the assumption.  Thus $G_P$ must fix the two out-neighbours of $P$ in $B_s(\Gamma_+)$.  A similar argument shows that $G_P$ must also fix the two in-neighbours of $P$ in $B_s(\Gamma_+)$.   Thus $G_P$ fixes all four neighbours of $P$ in $B_s(\Gamma_+)$ and, since $B_s(\Gamma_+)$ is connected, we see that $G_P$ acts trivially on $B_s(\Gamma_+)$.  But then $G_P=G_{\alpha_0,\ldots, \alpha_s}$ acts trivially on $\Gamma_+$ and $G_{\alpha_0}$ is finite, contrary to hypothesis.  Now we can conclude that $G$ acts transitively on the set of $s$-arcs for any $s\geq 1$, and we see that $\Gamma$ is a tree.

From Lemma~\ref{lem:infinte-arcs} it follows that $G$ acts transitively on the set of all 2-way infinite arcs in $\Gamma$. That in turn implies that $G$ acts 2-transitively on the set of ends of $\Gamma$.
\end{proof}

\begin{remark}
There are infinitely many examples of simple groups of the type described in the above theorem, see the paper by Caprace and Radu \cite[Remark A4]{CapraceRadu2020}.  
\end{remark}

Next we consider Case {\bf B}.

\begin{theorem}\label{thm:trivalentB}  Suppose $\Gamma$ is a connected trivalent graph and $G\leq \aut(\Gamma)$ acts vertex- and edge-transitively, but not arc-transitively, on $\Gamma$ and the stabilizers in $G$ of vertices are infinite.  Then $\Gamma$ is a tree. 

Furthermore, let $\Gamma_+$ denote the digraph that has the same vertex set as $\Gamma$ and the set of arcs is one of the arc-orbits of $G$ on $\Gamma$. Then $G$ acts highly-arc-transitively on $\Gamma_+$.
    In particular, $G$ fixes an end $\omega$ of $\Gamma$ and, if $G$ is a closed subgroup of $\aut(\Gamma)$, then $G$ acts transitively on $\Omega\Gamma\setminus \{\omega\}$.
\end{theorem}

\begin{proof}
Clearly $G$ has two orbits on the arcs of $\Gamma$.  Let $\Gamma_+$ be the digraph that has the same vertex set as $\Gamma$ and has one of the arc-orbits as a set of arcs.  We choose the orbit so that the in-degree is 1 and the out-degree is 2. Any cycle in  $\Gamma_+$ would have to be a directed cycle, since otherwise we would have a vertex with in-degree 2.  If $\alpha$ is a vertex in a directed cycle in $\Gamma_+$ then an automorphism fixing $\alpha$ and taking one of the outgoing arcs to the other will move our directed cycle to a different directed cycle that also includes $\alpha$.   This leads to a contradiction because the subdigraph consisting of these two cycles will then have a vertex with in-degree 2.  Thus $\Gamma$ cannot contain a cycle and $\Gamma$ is therefore a tree.

The part about the action on $\Gamma_+$ being highly-arc-transitive is proved in the same way as in the last theorem.  Suppose $G$ acts $s$-arc-transitively on $\Gamma_+$ but not $(s+1)$-arc-transitively.  Then $G$ acts vertex-transitively on the $s$-arc-digraph $B_s(\Gamma_+)$ but, as in the proof of the last theorem, we see that the stabilizer of a vertex in $B_s(\Gamma_+)$ must fix all three neighbouring vertices and thus the stabilizer in $G$ of a vertex in $B_s(\Gamma_+)$ acts trivially on the whole graph $B_s(\Gamma_+)$.   Thus the stabilizer in $G$ of a vertex in $\Gamma$ is a finite group.  Now we have reached a contradiction and conclude that $G$ must be highly-arc-transitive.

From Lemma~\ref{lem:infinte-arcs} it follows that $G$ acts transitively on the set of 2-way infinite-arcs in the digraph $\Gamma_+$.  Given a vertex $\alpha_0$ there is a unique arc $(\ldots, \alpha_{-1}, \alpha_0)$ and the end $\omega$ that contains the ray $\ldots, \alpha_{-1}, \alpha_0$ is fixed by the automorphism group.   Thus 
 $G$ fixes one point $\omega$ in the boundary and acts transitively on $\Omega\Gamma\setminus\{\omega\}$.
\end{proof}

\begin{corollary}\label{cor:tree}
Suppose $\Gamma$ is a connected trivalent graph.  Suppose $G\leq \aut(\Gamma)$ acts  
vertex- and edge-transitively on $\Gamma$ and the stabilizers in $G$ of vertices in $\Gamma$ are infinite.  Then $\Gamma$ is the 3-regular tree.
\end{corollary}

\begin{remark} $  $
\begin{enumerate}
    \item Theorem~\ref{thm:trivalentB} could also be proved by referring to \cite[Proposition 22]{ArnadottirLederleMoller2020a}.  That result says that if a group acts vertex- and arc-transitively on a locally finite digraph and the in- and out-degrees are coprime, then the action is highly-arc-transitive and the subdigraph spanned by the set of descendants of a vertex is a tree.
    \item In his study of non-discrete, vertex-transitive actions on the regular trivalent tree Nebbia gets the same conclusions about arc-transitivity as in the above theorems but he assumes from the start that the graph is a tree, see \cite[Proposition~3.1]{Nebbia2013}.
\end{enumerate}
\end{remark}

\section{The non-edge transitive case}

Now we turn our attention to Case {\bf C}, i.e.~$\Gamma$ is a connected trivalent graph and $G$ a subgroup of $\aut(\Gamma)$ that acts vertex-transitively on $\Gamma$ with two orbits on the edges and infinite vertex stabilizers.  As described in Section~\ref{sec:three-cases}, we think of the edges of our graph $\Gamma$ as being coloured red or blue according to which orbit they belong to.  Choose the colouring so that  each vertex is adjacent to precisely one red edge and precisely two blue edges.   An arc inherits its colour from the edge that gives rise to it.

\begin{definition}
  Let $S=(\alpha_0,\dots,\alpha_s)$ be an $s$-arc in $\Gamma$. Then $\alpha_0$ is called the \emph{head} and $\alpha_s$ the \emph{tail} of $S$.   The {\em reverse} of the arc $S=(\alpha_0,\dots,\alpha_s)$ is the arc $\overline{S}=(\alpha_s,\dots,\alpha_0)$.
  We call $S$ \emph{alternating} if consecutive arcs are in different $G$-orbits. If $(\alpha_0,\alpha_1)$ is red and $(\alpha_{s-1},\alpha_s)$ is blue we call $S$ an \emph{rb-alternating s-arc}; {\em rr-alternating s-arcs}, {\em br-alternating s-arcs} and {\em bb-alternating s-arcs} are defined in the obvious way.
\end{definition}

Note that rr-alternating and bb-alternating arcs have odd length and rb-alternating and br-alternating arcs have even length.  The partition of alternating $s$-arcs into rb-, rr-, br- and bb-alternating s-arcs is invariant under automorphisms of $\Gamma$.  

\begin{remark}
If $G$ acts transitively on the set of bb-alternating $s$-arcs, then $G$ acts transitively on  the set of rb-alternating $(s+1)$-arcs, the set of br-alternating $(s+1)$-arcs and the set of rr-alternating $(s+2)$-arcs.
\end{remark}

\begin{definition}
  An alternating $s$-arc $S'$ is a \emph{predecessor} of an alternating $s$-arc $S$ if there exists an alternating $(s+1)$-arc $(\alpha_0,\dots,\alpha_{s+1})$ with $S=(\alpha_0,\dots,\alpha_s)$ and $S'=(\alpha_1,\dots,\alpha_{s+1})$. We also say that $S$ is a \emph{successor} of $S'$.
  
  An alternating $s$-arc $S$ is said to be \emph{accessible} from an alternating $s$-arc $S'$ if there exists a finite sequence $S_0,\dots,S_n$ of alternating $s$-arcs such that $S=S_0$, $S'=S_n$ and for all $0 \leq i \leq n-1$ the $s$-arc $S_{i+1}$ is a predecessor or a successor of $S_i$.
  
  Similarly, an alternating $s$-arc $S'$ is a \emph{$2$-predecessor} of an alternating $s$-arc $S$ if there exits an alternating $(s+2)$-arc $(\alpha_0,\dots,\alpha_{s+2})$ with $S=(\alpha_0,\dots,\alpha_s)$ and $S'=(\alpha_2,\dots,\alpha_{s+2})$. In this situation we also say that $S$ is a \emph{$2$-successor} of $S'$.
  
  An alternating $s$-arc $S$ is said to be \emph{$2$-accessible} from an alternating $s$-arc $S'$ if there exists a finite sequence $S_0,\dots,S_n$ of alternating $s$-arcs such that $S=S_0$, $S'=S_n$ and for all $0 \leq i \leq n-1$ the $s$-arc $S_{i+1}$ is a $2$-predecessor or a $2$-successor of $S_i$.
\end{definition}

Note that accessibility and $2$-accessibility are equivalence relations on $s$-arcs.
The following is a generalization of (2.1) in \cite{Tutte1959}.

\begin{lemma}\label{lem:accessibility}  $ $
\begin{enumerate}
    \item Let $S$ be an alternating $s$-arc and $\overline{S}$ its reverse.
Any alternating $s$-arc is accessible from $S$ or $\overline{S}$.
   \item Let $S$ be an rr-alternating $s$-arc. Any $rr$-alternating $s$-arc is $2$-accessible from $S$ or $\overline{S}$.
\end{enumerate}
\end{lemma}

\begin{proof}
1.  Let $W$ denote the set of all alternating $s$-arcs that are accessible from $S$ or $\overline{S}$.  Clearly $S, \overline{S} \in W$.   Note that if an alternating $s$-arc $S'$ is accessible from some alternating $s$-arc $S$ then $\overline{S'}$ is accessible from $\overline{S}$.

Let $\alpha$ be a vertex that belongs to some $S'\in W$.   Assume that $S''$ is an alternating $s$-arc containing $\alpha$.
We can assume that $\alpha$ is the head of both $S'=(\alpha_0,\dots,\alpha_s)$ and $S''=(\beta_0,\dots,\beta_s)$. If $\alpha_1=\beta_1$, then, by repeatedly taking successors we can find an $s$-arc $(\gamma_0,\ldots, \gamma_{s-2},\alpha_0,\alpha_1)$ from which both $S'$ and $S''$ are accessible and since $S'\in W$ we see that $S''$ is also in $W$. If $\alpha_1 \neq \beta_1$, there is a case distinction. Let $\gamma$ be the third neighbour of $\alpha$. If both $(\alpha,\alpha_1)$ and $(\alpha,\beta_1)$ are blue, then there exists an alternating $s$-arc ending  with the red arc $(\alpha,\gamma)$ from which both $S'$ and $S''$ are accessible. If one of $(\alpha,\alpha_1)$ and $(\alpha,\beta_1)$ is red and the other is blue, then $S''$ is clearly accessible from $\overline{S'}$.

Now consider the set $A$ of all the vertices that belong to some alternating $s$-arc in $W$.  Assume now that $\alpha$ is a vertex in $\Gamma$ but not in $A$ and that $\alpha$ is adjacent to some vertex $\beta$ in $A$.
Since $\beta\in A$  we know that $\beta$ is the head of some alternating $s$-arc $T$ in $W$. Say the edge $\{\beta, \delta\}$ belongs to $T$.

If the edges $\{\alpha, \beta\}$ and $\{\beta, \delta\}$ have different colours, then the vertex $\alpha$ clearly belongs to a successor of $T$ and is thus in $A$.
Suppose now that the edges $\{\alpha, \beta\}$ and $\{\beta, \delta\}$ have the same colour, i.e.~both are blue. Say $\{\beta,\gamma\}$ is the red edge incident with $\beta$.  Then $\gamma$ belongs to a successor of $T$ and we can find an alternating $s$-arc $T'$ that has $\beta$ as its tail and contains the vertex $\gamma$.  Clearly $\alpha$ belongs to a successor of $T'$, hence $\alpha \in A$, contradiction. This concludes the proof of the first part of the lemma.

2.  For the second part of the lemma, let $S$ and $S'$ be two rr-alternating $s$-arcs in $\Gamma$.
By the first part there is a sequence $S_1, S_2, \ldots, S_k$ such that $S_1=S$ or $S_1 = \overline{S}$ and $S_k=S'$ and for all $i=1, \ldots, k-1$ the alternating $s$-arc $S_{i+1}$ is either a predecessor or a successor of $S_i$.   If it so happens that $S_{i+1}$ is a predecessor of $S_i$ and $S_{i+2}$ is a predecessor of $S_{i+1}$ then $S_{i+2}$ is a $2$-predecessor of $S_i$ and, similarly, if $S_{i+1}$ is a successor of $S_i$ and $S_{i+2}$ is a successor of $S_{i+1}$ then $S_{i+2}$ is a $2$-successor of $S_i$.
If $S_{i+1}$ is a predecessor of $S_i$ and $S_{i+2}$ is a successor of $S_{i+1}$ then we let $S'_{i+1}$ be a predecessor of $S_{i+1}$ and note that then $S'_{i+1}$ is a 2-successor of $S_i$ and $S_{i+2}$ is a 2-successor of $S_{i+1}$.  In the case that  $S_{i+1}$ is a successor of $S_i$ and $S_{i+2}$ is a predecessor of $S_{i+1}$ can be handled similarly.  Thus we can construct a sequence of $s$-arcs starting with $S$ or $\overline{S}$ and ending with $S'$ such that each arc, except the first one, is the 2-predecessor or 2-successor of the previous one.
\end{proof}

\begin{lemma}\label{lem:rr transitivity}
Assume $G$ acts transitively on the set of rr-alternating $s$-arcs, but not on the set of rr-alternating $(s+2)$-arcs. Then, $G$ acts regularly on the set of rr-alternating $s$-arcs.
In particular $G$ has finite vertex stabilizers.
\end{lemma}

\begin{proof}
We first show that $G$ has two orbits on the set of rr-alternating $(s+2)$-arcs. Let $S=(\alpha_0,\dots,\alpha_{s})$ be an rr-alternating $s$-arc.
Then there are exactly two rr-alternating $(s+2)$-arcs $S_1=(\alpha_0,\dots,\alpha_{s},\alpha_{s+1},\alpha_{s+2})$ and $S_1=(\alpha_0,\dots,\alpha_{s},\alpha'_{s+1},\alpha'_{s+2})$ extending $S$.
For every rr-alternating $s$-arc $(\beta_0,\dots,\beta_{s}, \beta_{s+1},\beta_{s+2})$ there exists an element $g \in G$ with $(\beta_0,\dots,\beta_{s})g=S$.  Then $g$ maps the $(s+2)$-arc $(\beta_0,\dots,\beta_{s}, \beta_{s+1},\beta_{s+2})$ to one of the 2-predecessor of $S$.  Thus $(\beta_0,\dots,\beta_s,\beta_{s+1},\beta_{s+2})$ lies in the orbit of exactly one of $S_1$ and $S_2$.

In particular, every element fixing an rr-alternating $s$-arc $S$ pointwise has to fix both of its rr-alternating 2-predecessors and 2-successors.
Inductively, we see that the pointwise stabilizer of $S$ has to fix all the vertices that are contained in any rr-alternating $s$-arc that is $2$-accessible from $S$.
By Lemma \ref{lem:accessibility} every $rr$-alternating $s$-arc is 2-accessible from $S$ and thus the pointwise stabilizer of $S$ fixes every vertex in the graph and is trivial.  Hence vertex stabilizers in $G$ are finite.
\end{proof}

The above lemmas imply the following theorem.

\begin{theorem}\label{thm:trivalentC}
Suppose $\Gamma$ is a connected trivalent graph and $G\leq \aut(\Gamma)$ acts vertex-transitively, but not edge-transitively, on $\Gamma$ and the stabilizers of vertices are infinite.  
    Then, for every $s \geq 1$, the group $G$ acts transitively on the set of all alternating $s$-arcs $(\alpha_0,\dots,\alpha_s)$ that start with an edge of a given colour.
\end{theorem}

\begin{proof}
By Lemma \ref{lem:rr transitivity} the group $G$ acts transitively on the set of rr-alternating and the set of bb-alternating $s$-arcs for all odd $s \geq 1$. Then it also acts transitively on the set of all rb-alternating and the set of br-alternating arcs $s$-arcs for all even $s \geq 2$.
\end{proof}

\begin{corollary}\label{cor:no-alternating-cycles}
Suppose $\Gamma$ is a connected trivalent graph and $G\leq \aut(\Gamma)$ acts vertex-transitively, but not edge-transitively, on $\Gamma$ and the stabilizers of vertices are infinite.  Let $(\alpha_0, \ldots, \alpha_s)$ be an alternating $s$-arc.  Then  $\alpha_0\neq\alpha_s$ and  $\alpha_0$ and $\alpha_s$ are not adjacent.
\end{corollary}

\begin{proof}
Let us first show that is impossible that $\alpha_0=\alpha_s$.  
Clearly $s\geq 3$.  By renumbering the vertices in the cycle formed by the vertices $\alpha_0, \ldots, \alpha_s$ we may assume that the edge $\{\alpha_{s-1}, \alpha_s\}$ is blue.  Then there is a vertex $\beta$ such that $\{\alpha_{s-1}, \beta\}$ is also a blue edge and $\beta\neq \alpha_s$.  By the last theorem there exists an element $g\in G$ that takes the alternating $s$-arc $(\alpha_0, \ldots, \alpha_{s-1}, \alpha_s)$ to the alternating $s$-arc $(\alpha_0, \ldots, \alpha_{s-1}, \beta)$, but that is clearly impossible.

Suppose that $(\alpha_0, \ldots, \alpha_s)$ is an alternating $s$-arc such that $\alpha_0$ and $\alpha_s$ are adjacent.   By the above the vertices $\alpha_0, \ldots, \alpha_s$ are all distinct.   If the edge $\{\alpha_0, \alpha_s\}$ is red then $(\alpha_s, \alpha_0, \ldots, \alpha_s)$ would be an alternating $(s+1)$-arc contradicting what is shown above.  Thus the edge $\{\alpha_0, \alpha_s\}$ must be blue.  If the edge $\{\alpha_0, \alpha_1\}$ is red then $(\alpha_s, \alpha_0, \ldots \alpha_s)$ would be an alternating $(s+1)$-arc and that is impossible, and if the edge $\{\alpha_{s-1}, \alpha_s\}$ is red then $(\alpha_0, \ldots, \alpha_s, \alpha_0)$ would be an alternating $(s+1)$-arc.  Hence we see that both the edges $\{\alpha_0, \alpha_1\}$ and $\{\alpha_{s-1}, \alpha_s\}$ must be blue.  Let $\beta$ be a vertex, distinct from $\alpha_s$, such that $\{\alpha_{s-1}, \beta\}$ is a blue edge.  Note that, by the above $\alpha_s\neq \alpha_1$ and $\beta\neq \alpha_1$.  Let $g$ be an element in $G$ taking the alternating $s$-arc $(\alpha_0, \ldots, \alpha_{s-1}, \alpha_s)$ to the alternating $s$-arc $(\alpha_0, \ldots, \alpha_{s-1}, \beta)$.  Then $\{\alpha_0, \alpha_s\}g=\{\alpha_0, \beta\}$ is a blue edge and $\alpha_0$ is the end-vertex of 3 distinct blue edges $\{\alpha_0, \alpha_1\}, \{\alpha_0, \alpha_2\}$ and $\{\alpha_0, \beta\}$, which is impossible.   We have reached a contradiction and our proof is complete.
\end{proof}

\begin{corollary}
Suppose $\Gamma$ is a connected trivalent graph and $G\leq \aut(\Gamma)$ acts vertex-transitively, but not edge-transitively, on $\Gamma$ and the stabilizers of vertices are infinite.  Then $\Gamma$ contains an infinite alternating line and every alternating $s$-arc is a part an infinite alternating line.  
\end{corollary}

\begin{proof}
It is clear that every alternating $s$-arc can be extended to a 2-way infinite alternating arc.  By Corollary~\ref{cor:no-alternating-cycles} all the vertices in this infinite alternating arc must be distinct and thus we have an infinite alternating line.
\end{proof}

The argument used to prove Lemma~\ref{lem:infinte-arcs} can be adapted to show the following.

\begin{corollary}\label{cor:transitive-on-lines}
Suppose $\Gamma$ is a connected, trivalent graph and $G\leq \aut(\Gamma)$ acts vertex-transitively, but not edge-transitively, on $\Gamma$ and the stabilizers of vertices are infinite.  If $\ldots, \alpha_{-1}, \alpha_0, \alpha_1, \alpha_2, \ldots$ and $\ldots, \beta_{-1}, \beta_0, \beta_1, \beta_2, \ldots$ are two infinite alternating lines such that the edges $\{\alpha_0, \alpha_1\}$ and $\{\beta_0, \beta_1\}$ have the same colour, then there exists in $g \in G$ such that $\alpha_i g=\beta_i$ for all $i$.  
\end{corollary}

\section{2-ended trivalent graphs}

In this section we classify connected, 2-ended, trivalent graphs such that the automorphism group is vertex-transitive and the stabilizers of vertices are infinite.  

The argument used in the proof of the following lemma is somewhat reminiscent of arguments found in \cite{Moller2002} and the notation is chosen to reflect this similarity.  This lemma will be used again in Section~\ref{sec:Scale}.  

\begin{lemma}\label{lem:trivalent-two-ends}
Suppose $\Gamma$ is a connected, vertex-transitive, trivalent graph and $G\leq \aut(\Gamma)$ such that the conditions in Case {\bf C} hold.     Let $\ldots, \alpha_{-1}, \beta_{-1}, \alpha_0, \beta_0, \alpha_1, \beta_1, \alpha_2,\ldots$ be an alternating line in $\Gamma$ such that the edges of the type $\{\alpha_i, \beta_i\}$ are red and the edges of the type $\{\beta_i, \alpha_{i+1}\}$ are blue.  If there is a constant $C$ such that $|\alpha_i G_{\alpha_0}|\leq C$ for all $i\geq 1$, then $\Gamma$ has exactly two ends.
\end{lemma}

\begin{proof}   By Corollary~\ref{cor:transitive-on-lines} there exists an element $g\in G$ such that $\alpha_ig=\alpha_{i+1}$ and $\beta_i g=\beta_{i+1}$ for all $i$.
Set $U=G_{\alpha_0}$.   Define $U_{-\infty, i}$ as the subgroup of $G$ fixing pointwise the ray $\ldots, \alpha_{i-1}, \beta_{i-1}, \alpha_{i}, \beta_{i}$. These subgroups are all conjugate via powers of $g$.  Now define $U_{++}$ as the subgroup $\bigcup_{i\in \ZZ} U_{-\infty, i}$ and $G_{++}=\langle U_{++}, g\rangle$. Clearly 
$$G_{++}=\{h\in G\mid \mbox{there exist }m, n\in \ZZ\mbox{ such that }(\ldots, \alpha_m, \beta_m)h=(\ldots, \alpha_n, \beta_n)\}.$$
Note that $g^{-1} U_{++} g= U_{++}$.   Let $\Gamma_{++}$ denote the subgraph that has vertex set $\alpha_0 G_{++}\cup \beta_0G_{++}$ and edge set $\{\alpha_0, \beta_0\}G_{++}\cup \{\beta_0, \alpha_1\}G_{++}$.  Our aim is to show that $\Gamma_{++}$ is equal to $\Gamma$.  The graph $\Gamma_{++}$ is connected.  The group $G_{++}$ has at most two orbits on the vertex set of $\Gamma_{++}$  and also at most two orbits on the edge set. It follows from the transitivity on alternating lines (see Corollary~\ref{cor:transitive-on-lines} above)  that all the vertices in the orbit $\beta_i G_{++}$ have degree $3$ in the graph $\Gamma_{++}$.   

Suppose $n$ is a number such that $2^n>C$.  There are $2^n$ alternating $2n$-arcs having $\alpha_0$ as their initial vertex and starting with the red edge $\{\alpha_0, \beta_0\}$.  From Corollary~\ref{cor:transitive-on-lines} we see that the group $(U_{++})_{\alpha_0}$ acts transitively on the set of these arcs.  But the orbit $\alpha_n (U_{++})_{\alpha_0}$ has fewer than $2^n$ elements and thus there is some alternating $2n$-arc in $\Gamma_{++}$ of the form $(\alpha_0, \beta_0', \ldots, \beta'_{n-1}, \alpha_n)$ that is different from the $2n$-arc $(\alpha_0, \beta_0, \ldots, \beta_{n-1}, \alpha_n)$.  Note that it is impossible that $\alpha_i=\alpha'_i$ for all $i$.  Let $i$ be the biggest number such that $\alpha_i\neq \alpha'_i$.  Then $\beta_i\neq\beta'_i$ and $\beta'_{i+1}=\beta_{i+1}$.  Hence the vertices $\beta_i, \beta'_i$ and $\beta_{i+1}$ are all distinct and all of them are neighbours of $\alpha_{i+1}$ (recall that the edges $\{\beta_i, \alpha_{i+1}\}$ and $\{\beta'_i, \alpha_{i+1}\}$ are both blue but the edge $\{\alpha_{i+1}, \beta_{i+1}\}$ is red).  Thus the vertex $\alpha_{i+1}$ also has degree 3 in $\Gamma_{++}$.  Hence the graph $\Gamma_{++}$ is regular with degree 3.   Since the graph $\Gamma$ is trivalent and connected, we see that $\Gamma_{++}=\Gamma$.  

The orbits $\alpha_i U_{++}$ are all finite and each orbit has size at most $C$.  The same holds true for the orbits $\beta_i U_{++}$.  We also see that $(\alpha_i U_{++})g=\alpha_{i+1} U_{++}$ and similarly that $(\beta_i U_{++})g=\beta_{i+1} U_{++}$.  Hence $g$ has at most $2C$ orbits on $\Gamma$.   A result of Jung and Watkins \cite[Theorem~5.12]{JungWatkins1984} says that a connected vertex transitive graph that has an automorphism with only finitely many orbits has just two ends.
\end{proof}

\begin{remark}
From the argument above we see that it is enough to assume that there exists some positive integer $n$ such that $|\alpha_n G_{\alpha_0}|<2^n$ to get the conclusion that $\Gamma$ has exactly two ends.
\end{remark}

The next result is a classification of connected, vertex-transitive, trivalent graphs with two ends such that the stabilizers in the automorphism group are infinite.  Some preliminary work is needed before we can state the theorem.

 In \cite[Corollary 16]{MollerPotocnikSeifter2019} highly-arc-transitive digraphs with two ends and prime in- and out-degree are classified:  Let $\Delta_p$ be the digraph with vertex set $\ZZ\times \{1, \ldots, p\}$ and arc set the set of all pairs $((i, j), (i+1, j'))$ with $i\in \ZZ$ and $j, j'\in \{1, \ldots, p\}$.  Any highly-arc-transitive digraph with two ends and  in- and out-degree equal to $p$ is isomorphic to $\Delta_p$ or one of its $s$-arc-digraphs $B_s(\Delta_p)$.   

These digraphs all have the property that they are isomorphic to their reverse digraph.  Thus one can apply the construction described in Part 6 of Example~\ref{ex:the-cases} to $B_s(\Delta_2)$ and get a trivalent graph $\Theta_{s}:= (B_s(\Delta_2))_*$ such that its automorphism group satisfies the conditions in Case {\bf C}.  Let $\Theta_0$ denote the digraph we get from $\Delta_2$.  If $\Delta$ is isomorphic to $\Delta_2$ then $\Gamma$ is isomorphic to $\Theta_0$, but if $\Delta$ is isomorphic to $B_s(\Delta_2)$ for some $s\geq 1$ then $\Gamma$ is isomorphic to $\Theta_s$. 

\begin{theorem}\label{thm:2-ends}
Suppose $\Gamma$ is a connected, vertex-transitive, trivalent graph with two ends.  Suppose the stabilizers of vertices in $\aut(\Gamma)$ are infinite.  Then $\Gamma$ is isomorphic to $\Theta_s$ for some $s\geq 0$.  
\end{theorem}

\begin{proof}
Continue with the setup in the proof of the previous lemma with $L$ denoting the alternating line with vertex set $\{\ldots, \alpha_{0}, \beta_{0}, \alpha_1, \beta_1, \ldots\}$ .  We aim to construct on the basis of $\Gamma$ a connected highly-arc-transitive digraph such that all vertices have in-degree 2 and out-degree 2.  
Consider the digraph ${\Gamma}_+$ that has the same vertex set as $\Gamma$ and the arc set is the set  $(\alpha_0, \beta_0)G_{++}\cup (\beta_0, \alpha_1)G_{++}$.   First we show that it is impossible that there are vertices $\gamma$ and $\delta$ such that both $(\delta, \gamma)$ and $(\gamma, \delta)$ are arcs in ${\Gamma}_+$.   The group $G_{++}$ fixes both ends of $\Gamma$ but if there was an element $f\in G_{++}$ that would transpose two adjacent vertices, say that the edge between them is red, then there would be an element $h$ in $G_{++}$ that would transpose the vertices $\alpha_0$ and $\beta_0$ and map the line $L$ to itself such that $\alpha_i h=\beta_{-i}$ and $\beta_i g=\alpha_{-i}$ and thus would not fix the two ends of $\Gamma$, contradicting our assumptions.  Contract now all the arcs in ${\Gamma}_+$ that come from red edges in $\Gamma$ and we get a digraph $\Delta$ with two ends where the in- and out-degrees of every vertex are 2.  By Theorem~\ref{thm:trivalentC} this digraph is highly-arc-transitive and thus isomorphic to $B_s(\Delta_2)$ for some $s\geq 2$.   Then $\Gamma$ is isomorphic to $\Theta_s$.  
\end{proof}

\section{Connection with totally disconnected,\\ locally compact groups}\label{sec:Scale}

The study of totally disconnected, locally compact groups has become an active field in recent years, largely due to the efforts of George Willis and his coworkers, see e.g.~\cite{Willis1994} and \cite{CapraceReidWillis2017}.  The connection with group actions on graphs uses the {\em Cayley--Abels graph}.  

Let $G$ be a compactly generated, totally disconnected, locally compact group.  If $G$ acts vertex-transitively on a connected, locally finite graph $\Gamma$ such that the stabilizers of vertices are compact, open subgroups of $G$, then we say that $\Gamma$ is a {\em Cayley--Abels graph} for $G$.  A Cayley--Abels graph for $G$ can be constructed by starting with a compact generating set $C$ and a compact open subgroup $U$ of $G$ (such a subgroup always exists by an theorem of van Dantzig, \cite{vanDantzig1936}).  Then we form the Cayley graph of $G$ with respect to $C$ and define $\Gamma$ as the quotient graph with respect to the left action of $U$.  Note that the vertex set of $\Gamma$ is the set of right cosets of the subgroup $U$.  For further information and another construction see the survey paper \cite{Moller2010}.  Define $\mv(G)$ as the lowest possible degree of a Cayley--Abels graph for $G$.  This concept is the main topic of discussion in \cite{ArnadottirLederleMoller2020a}.  

In \cite{Willis1994}, Willis defined the concepts of {\em tidy subgroups} and the {\em scale function}.  In this work we will only discuss the scale function and we use as definition a formulation from Willis's later paper \cite{Willis2001}.  The {\em scale function} on a totally disconnected, locally compact group $G$ is the function $s:G\to \ZZ_+$ defined by the formula
$$s(g)=\min\{|U:U\cap g^{-1}Ug| \mid U \text{ compact open subgroup of } G\}.$$

A totally disconnected, locally compact group is said to be {\em uniscalar} if $s(g)=1$ for all $g\in G$.  Let $U$ be a compact, open subgroup of $G$ and consider the action of $G$ on the set of right cosets $\Omega=G/U$.  Set $\alpha=U$ and think of $\alpha$ as a point in $\Omega$.  Then 
$$s(g)=\lim_{n\to \infty} |(\alpha g^n) G_\alpha|^{1/n},$$
and, furthermore, $s(g)=1$ if and only if there is a constant $C$ such that $|(\alpha g^i)G_\alpha|\leq C$ for all $i=0, 1, 2, \ldots$ (see \cite[Corollary 7.8]{Moller2002}).

The connection between totally disconnected, locally compact groups and group actions on graphs works in both directions.  When $G$ is a group acting on a set $\Omega$, e.g.~the automorphism group of a graph $\Gamma$ acting on the vertex set $\V\Gamma$,  we can endow $G$ with the {\em permutation topology}, see for instance \cite{Woess1991} and \cite{Moller2010}.   One way to define the permutation topology is to say that a neighbourhood basis of the identity is formed by the family of all subgroups of the form $G_{(\Phi)}$, where $\Phi$ ranges over all finite subsets of $\Omega$.  If the group $G$ already has a topology and the stabilizer $G_\alpha$ of a point $\alpha\in \Omega$ is open, then the permutation topology is a subset of the topology on $G$.  
The convergence defined in Section~\ref{sec:Convergence} is the same as convergence in this topology.  If $G$ is a closed subgroup of the automorphism group of a locally finite graph $\Gamma$, then $G$ is a totally disconnected, locally compact group, see  \cite[Lemma 1]{Woess1991} and \cite[Lemma 2.2]{Moller2010}.

\begin{lemma}\label{lem:uniscalar-two-ends}
Let $G$ be a totally disconnected, locally compact group.  Suppose $\Gamma$ is a trivalent Cayley--Abels graph for $G$ such that the conditions in Case {\bf C} are satisfied.  If the group $G$ is uniscalar, then $\Gamma$ has two ends and $G$ has a compact, open, normal subgroup.
\end{lemma}

\begin{proof}   Let $\ldots, \alpha_{-1}, \beta_{-1}, \alpha_0, \beta_0, \alpha_1, \beta_1, \alpha_2,\ldots$ be an alternating line in $\Gamma$ such that the edges of  type $\{\alpha_i, \beta_i\}$ are red and  edges of the type $\{\beta_i, \alpha_{i+1}\}$ are blue. By Corollary~\ref{cor:transitive-on-lines}, there exists $g\in G$ such that $\alpha_i g=\alpha_{i+1}$ and $\beta_i g=\beta_{i+1}$.  As mentioned above, the assumption that $s(g)=1$ implies that there is a constant $C$ such that $C\geq |(\alpha_0 g^n) G_{\alpha_0}|=|\alpha_n G_{\alpha_0}|$ for all $n$ and now we see from Lemma~\ref{lem:trivalent-two-ends} that $\Gamma$ has just two ends.  Then there is a compact open normal subgroup $K$ such that $G/K$ is either isomorphic to $\ZZ$ or the infinite dihedral group $D_\infty$, see \cite[Proposition 3.2]{MollerSeifter1998}. 
\end{proof}

\begin{theorem}\label{thm:not-uniscalar}
Suppose $G$ is a compactly generated, totally disconnected, locally compact group that does not have a compact, open, normal subgroup.  If $\mv(G)=3$ then $G$ is not uniscalar.   
\end{theorem}

\begin{proof}   Let $\Gamma$ be a trivalent Cayley--Abels graph for $G$.  We consider separately what happens in Cases {\bf A}, {\bf B} and {\bf C}.

Let us first look at Case {\bf A}.  Consider an infinite line $\ldots, \alpha_{-1}, \alpha_0, \alpha_1, \alpha_2,\ldots$.  Let $g\in G$ be an element such that $\alpha_i g=\alpha_{i+1}$ for all $i$.  Then $\alpha_0 g^n=\alpha_n$ and by Theorem~\ref{thm:trivalentA} we see that $|\alpha_n G_{\alpha_0}|= 3\cdot 2^{n-1}$ and then 
$$s(g)=\lim_{n\to \infty} |(\alpha g^n) G_\alpha|^{1/n}=\lim_{n\to \infty} \big(3\cdot 2^{n-1}\big)^{1/n}=2.$$
Hence $G$ is not uniscalar.

In Case {\bf B} we let $\Gamma_+$ be the digraph defined in the proof of Theorem~\ref{thm:trivalentB}.  Suppose that $(\ldots, \alpha_{-1}, \alpha_0, \alpha_1, \ldots)$ is a 2-way infinite arc in $\Gamma_+$ and $g\in G$ acts like a translation on this arc such that $\alpha_i g=\alpha_{i+1}$ for all $i$.  The fact that $G$ acts highly-arc-transitively on $\Gamma_+$ implies that if $n\geq 0$ then $|\alpha_n G_{\alpha_0}|= 2^{n}$.  Thus 
$$s(g)=\lim_{n\to \infty} |(\alpha g^n) G_\alpha|^{1/n}=\lim_{n\to \infty} \big(2^{n}\big)^{1/n}=2.$$

And, finally, it is the case when the action of $G$ on $\Gamma$ satisfies the conditions in Case {\bf C}.  From Lemma~\ref{lem:trivalent-two-ends} we see that if $G$ is uniscalar and has a trivalent Cayley--Abels graph satisfying the conditions in Case {\bf C} then $\Gamma$ has two ends and by Theorem 44 in \cite{ArnadottirLederleMoller2020a} it follows that $\mv(G)=2$.  Therefore, if $\mv(G)=3$ then $G$ can not be uniscalar.  
\end{proof}

Let $G$ be a totally disconnected, locally compact group and $g\in G$.  
From the definition of the scale function we see that $s(g)=1$ if and only if $g$ normalizes some compact, open subgroup of $G$, and the group $G$ is uniscalar if and only if for every element of $G$ there is some compact open subgroup $G$ normalized by $g$.   If $G$ has a compact, open, normal subgroup then $G$ is clearly uniscalar.  Bhattacharjee and Macpherson \cite[Section 3]{BhattacharjeeMacpherson2003} (following up on work by Kepert and Willis, \cite{KepertWillis2001}), constructed an example of a compactly generated, totally disconnected, locally compact group that has no compact, open, normal subgroup, but every element normalizes some compact open subgroup.  On the other hand Gl\"ockner and Willis have shown in \cite{GlocknerWillis2001} that a compactly generated, uniscalar $p$-adic Lie group has a compact, open, normal subgroup.

\begin{corollary}\label{cor:uniscalar implies nearly discrete}
Let $G$ be a compactly generated, totally disconnected, locally compact group having a trivalent Cayley--Abels graph.
If every $g \in G$ normalizes a compact open subgroup of $G$ (i.e.~$G$ is uniscalar) then $G$ has a compact, open, normal subgroup.
\end{corollary}

\begin{proof}
If the action of $G$ on $\Gamma$ is discrete, then there is nothing more to be done because the kernel of the action is a compact, open, normal subgroup.   But if the action is not discrete, then the graph $\Gamma$ must have precisely two ends and then there is a compact, open, normal subgroup $K$ such that $G/K$ is either isomorphic to $\ZZ$ or the infinite dihedral group $D_\infty$, see \cite[Proposition 3.2]{MollerSeifter1998}.  
\end{proof}

\section{Trofimov's result}

For a graph $\Gamma$ we let $\Gamma_n$ denote the the graph that has the same vertex set as $\Gamma$ and two distinct vertices $\alpha$ and $\beta$ are adjacent in $\Gamma_n$ if and only if $d_\Gamma(\alpha, \beta)\leq n$.  

\begin{definition} {\rm (\cite{Cornulier2019})}
Let $\Gamma$ be a graph.  We say that $\Gamma$ {\em essentially includes a tree} if  the graph $\Gamma_n$ contains the 3-regular tree.
\end{definition}

An action of a group $G$ on a set $\Omega$ is said to be {\em nearly discrete} if there is a $G$-congruence $\sigma$ on $\Omega$ with finite equivalence classes such that if $K$ is the kernel of the action of $G$ on $\Omega/\sigma$ then the action of $G/K$ on $\Omega/\sigma$ is discrete.   When $G$ acts as a closed group on a locally finite graph then the kernel is a compact, open subgroup in the permutation topology (see \cite[Fact 5.6]{Cornulier2019}).

In his paper from 1984, \cite{Trofimov1984}, Trofimov considers the following question:

\smallskip

{\em Is it true that if $\Gamma$ is a locally finite connected graph and $G$ is a vertex-transitive subgroup of $\aut(\Gamma)$, then either the action is nearly discrete or the graph $\Gamma$ essentially includes a tree?}

\smallskip

Recently Cornulier,  \cite{Cornulier2019}, has constructed an example of a vertex-transitive, locally finite graph that does not essentially include a tree and the action of its automorphism group is not nearly discrete, thereby giving a negative answer to Trofimov's  question.   But Trofimov had shown that there cannot be a counterexample of degree 3. Our methods give a short proof of that result.

\begin{theorem}{\rm (\cite[Theorem 3.1]{Trofimov1984})}\label{thm:Trofimov}
Let $\Gamma$ be a vertex-transitive, trivalent graph and $G=\aut \Gamma$.  Then $G$ has a compact, normal subgroup $N$ such that the stabilizers in $\aut\, \Gamma/N$ of vertices in $\Gamma/N$ are finite, or $\Gamma_2$ contains a subgraph isomorphic to the 3-regular tree.  
\end{theorem}

\begin{proof}
If the stabilizers of vertices in $G$ are finite, then we can take $N$ as the trivial group.  Thus we may assume that the stabilizers in $G$ of vertices in $\Gamma$ are infinite.   If $G$ is edge-transitive then Corollary~\ref{cor:tree} says that $\Gamma$ is a tree.  Hence we may assume that $G$ is not transitive on the edges of $\Gamma$ and that the conditions in Case {\bf C} are satisfied.  Let $\ldots, \alpha_{-1}, \beta_{-1}, \alpha_0, \beta_0, \alpha_1, \beta_1, \ldots$ be an alternating line in $\Gamma$ such that the edges $\{\alpha_i, \beta_i\}$ are red and the edges $\{\beta_i, \alpha_{i+1}\}$ are blue.  If there is a positive integer $n$ such that $|\alpha_n G_{\alpha_0}|<2^n$, then it follows from Lemma~\ref{lem:uniscalar-two-ends} and the remark following its proof that $\Gamma$ has exactly two ends and the action is nearly discrete.

Since $|\alpha_n G_{\alpha_0}|\leq 2^n$, we are now left to consider the case where $|\alpha_n G_{\alpha_0}|= 2^n$ for every positive integer $n$.  First note that $\{\alpha_i, \alpha_{i+1}\}$ is an edge in the graph $\Gamma_2$.  We consider the subgraph $\Delta$ of $\Gamma_2$ with vertex set $\bigcup_{i=0}^\infty \alpha_i G_{\alpha_0}$ and edge set $\bigcup_{i=0}^\infty \{\alpha_i, \alpha_{i+1}\} G_{\alpha_0}$.  In this graph the vertex $\alpha_0$ has degree 2 and every other vertex has degree 3.  The set of vertices in $\Delta$ at distance $n$ from $\alpha_0$ is equal to $\alpha_n G_{\alpha_0}$ and since $|\alpha_n G_{\alpha_0}|= 2^n$ we conclude that $\Delta$ is the infinite rooted binary tree.  When we apply the same argument to $G_{\alpha_{-1}}$ and the ray $\alpha_{-1}, \alpha_{-2}, \ldots$ we find another copy of the rooted binary tree inside $\Gamma_2$.  This second tree has root $\alpha_{-1}$ and is disjoint from the first one and since the two roots, $\alpha_{-1}$ and $\alpha_0$ are adjacent in $\Gamma_2$ these two tree together with the edge $\{\alpha_{-1}, \alpha_0\}$ gie a copy of the 3-regular tree. 

Hence, $\Gamma_2$ contains a subdivision of the 3-regular tree.
\end{proof}

\begin{remark}
The authors of this paper have not had access to Trofimov's original paper \cite{Trofimov1984} and have their information about this result from the review in MathSciNet and the later paper \cite{Trofimov2007}. 
\end{remark}

\medskip

Combining Proposition~\ref{prop:locally-dihedral} with Corollary~\ref{cor:uniscalar implies nearly discrete} and Theorem~\ref{thm:Trofimov} one gets:

\begin{proposition}{\rm (\cite[Example 5.5]{Trofimov2007})}
Let $G$ be a group that acts vertex-transitively on a locally finite, connected graph $\Gamma$ of degree $d$.  Assume that $G$ acts locally like the dihedral group with $2d$ elements in its usual action on a set with $d$ elements.   Then, either the action is nearly discrete or the graph $\Gamma$ essentially includes a tree.  
\end{proposition}

\bibliographystyle{abbrv}
\bibliography{references}

\end{document}